\documentclass{amsart}
\usepackage{amsmath,amsthm, amssymb,mathtools,mdwlist, MnSymbol,mathrsfs}

\newtheorem{thm}{Theorem}[section]

   \makeatletter\let\c@cor\c@thm\makeatother

\newtheorem{prop}{Proposition}[section]
   \makeatletter\let\c@prop\c@thm\makeatother
   
\newtheorem*{prop*}{Proposition}

\newtheorem{lem}{Lemma}[section]
   \makeatletter\let\c@lem\c@thm\makeatother

\theoremstyle{definition}

\newtheorem{defn}{Definition}[section]
   \makeatletter\let\c@defn\c@thm\makeatother

\newtheorem{eg}{Example}[section]
   \makeatletter\let\c@eg\c@thm\makeatother

\theoremstyle{remark}

   \makeatletter\let\c@rmk\c@thm\makeatother

\usepackage[all]{xy}
\usepackage{hyperref}
\usepackage{tikz}
\usepackage{tikz-3dplot}
\usetikzlibrary{backgrounds}
\usepackage{subcaption}

\usepackage{mathrsfs}

\newcommand{\op}{\mathscr{O}}
\newcommand{\mop}{\mathbb{O}}
\newcommand{\opt}{\mathscr{P}}
\newcommand{\mopt}{\mathbb{P}}

\newcommand{\cM}{\mathcal{M}}
\newcommand{\cC}{\mathcal{C}}
\newcommand{\cW}{\mathcal{W}}
\newcommand{\cD}{\mathcal{D}}

\newcommand{\fa}[2]{{#1}({#2})}
\newcommand{\malg}{algebra}
\newcommand{\malgs}{algebras}
\newcommand{\oalg}{algebra}
\newcommand{\oalgs}{algebras}

\newcommand{\otop}[2]{{#1}\otimes_\op{#2}}
\newcommand{\unit}{{\epsilon}}

\newcommand{\sset}{s\mathbf{S}\text{et}}
\newcommand{\cat}{\mathbf{C}\text{at}}
\let\xto\xrightarrow
\everymath{\displaystyle}

\DeclareMathOperator{\id}{\mathrm{id}}

\begin{document}
	\title{Inverting operations in operads}
    
    \author[M. Basterra]{Maria Basterra}
\address{Department of Mathematics, University of New Hampshire,  Durham, NH 03824, USA}
\email{basterra@unh.edu}
\author[I. Bobkova]{Irina Bobkova}
\address{School of Mathematics, Institute for Advanced Study, Princeton, NJ 08540, USA}
\email{irinabobkova@gmail.com}
\author[K. Ponto]{Kate Ponto}
\address{Department of Mathematics, University of Kentucky, Lexington, KY 40506, USA}
\email{kate.ponto@uky.edu}
\author[U. Tillmann]{Ulrike Tillmann}
\address{Mathematical Institute, University of Oxford, Oxford OX2 6GG, UK}
\email{tillmann@maths.ox.ac.uk}
\author[S. Yeakel]{Sarah Yeakel}
\address{Department of Mathematics, University of Maryland, College Park, MD 20742, USA}
\email{syeakel@math.umd.edu}

\date{\today}

\begin{abstract}
We construct a localization for operads with respect to one-ary operations based on the Dwyer-Kan hammock localization \cite {DwyerKan}.  
For an operad $\op$ and a sub-monoid of one-ary operations $\cW$ we associate 
an operad $L\op $ and a canonical map $\op \to L\op $ which takes elements in  $\cW$ 
to  homotopy invertible operations. Furthermore, we give a functor  from the category of $\op$-algebras to the category of $L \op$-algebras satisfying an appropriate universal property.
\end{abstract}



\maketitle

\section{Introduction}

\vskip .1in
If $\op $ is an operad in simplicial sets  with $n$-ary operations  $\op(n)$ and  
$X$ is an algebra over $\op$, each $n$-ary operation gives a map 
 \[X^n \to X.\] 
Unless $X$ is quite trivial, we do not expect these maps to be invertible for $n\neq 1$. 
In contrast,  one might like to insist that (at least some of) the one-ary operations are invertible (up to homotopy). 
To facilitate the study of such algebras, we seek an operad that encodes such information. 
This leads us to search for a good definition of localization for an operad $\op$ with respect to a submonoid $\cW \subset \op(1)$ of one-ary operations. 

\vskip .1in
Localizations have been much studied in the literature, particularly in the context of model categories.  
An especially useful and well-studied construction of the localization of a category is the hammock localization of Dwyer and Kan \cite {DwyerKan}. 
We propose a variant of their construction where we consider not only hammocks of string type, 
that is a sequence of right and left pointing arrows, but also of tree type where left and right pointing arrows are assembled in a tree. 
This seems a necessary complication so that operad composition is well defined and associative after localization. 

Indeed, the complication arises because hammock localization does not preserve the monoidal structure of a category. 
It is well-known that operads (with an action of the symmetric group) correspond to 
strict (symmetric) monoidal categories with object set $\mathbb N $ where Hom-sets 
between two objects $a$ and $b$ are monoidally generated from Hom-sets with source 1.
Since the hammock localization of \cite {DwyerKan} does  not preserve the  monoidal structure, the outcome does not readily provide an operad.

\vskip .1in
The proposed tree hammock localization 
\[
  L^{TH}_\cW (\op) \quad (\text{or simply  }  L\op)
\] 
of $\op$ is functorial in the pair $(\op, \cW) $ and every operation in $\cW$ is invertible up to homotopy in $L\op$; 
see Lemmas \ref{op_loc_functor} and \ref{op_loc_we} and \autoref{w_action}.
Furthermore the derived tensor product defines a functor from $\op$-algebras to $L\op$-algebras satisfying 
a natural universal property with respect to $\op$-algebra maps to $L\op$-algebras; see \autoref{localization_alg}. 

\vskip .1in 
Our study has been motivated by considerations in topological and conformal quantum field theory.  
From Atiyah and Segal's axiomatic point of view, this is the study of symmetric monoidal functors from suitably defined cobordism categories.
A particularly well-studied theory is the field theory modeled by the 1+1 dimensional cobordism category 
where the objects are disjoint unions of circles and the morphisms are (oriented) surfaces with boundary.
Often one is led to the question of how the theory behaves stably, and more or less equivalently, when the operation defined by the torus is invertible. 
More generally there has been much recent interest in invertible topological field theories in the context of the study of anomalies and topological phases.  
See, for example, \cite{freedhopkins} and the references therein.

Much of topological field  theory is captured when restricting to the (maximal) symmetric monoidal sub-category of the cobordism category corresponding to an operad. 
Therefore, the study of the surface operad is essential to the study of 1+1 dimensional topological field theory and conformal field theory. 
It has many homotopy equivalent models; one, denoted by $\cM$, is built as a subcategory of Segal's category of Riemann surfaces \cite{Segal} 
and we will keep this operad in mind as an example.
This was also the motivating example for our study \cite {BBPTY} of operads with homological stability, compare \cite {T}.
Indeed our discussion there led us to consider localizations of the operad $\mathcal M$ 
in an attempt to answer an old question of Mike Hopkins\footnote{Stringy Topology in Morelia, Morelia, Mexico. 2006}. 

\subsection*{Content}
In Section \ref{sec:operads}, we review the definition of an operad and then, in Section \ref{sec:smc}, 
we  characterize them as a subcategory in the category of strict symmetric monoidal categories. 
Next, in  Section \ref{sec:hammock}, we recall the definition and some properties of the standard hammock localization for categories from \cite {DwyerKan}.  
Our main new construction is the tree hammock localization in Section \ref{sec:loc_operads} where we also study  several important properties. 
In Section \ref{sec:loc_alg} we provide a functor from algebras over a given operad to algebras over an associated localized operad.

\subsection*{Acknowledgments}
The present paper represents part of the authors' Women in Topology project, the main part of which \cite{BBPTY} will be published elsewhere. 
We  thank the organizers of the Women in Topology Workshop for providing a wonderful opportunity for collaborative research and  
the Banff International Research Station for supporting the workshop. We also thank the Association for Women in Mathematics for partial travel support.  

The fourth author would also like to thank Mikhail Kapranov for sharing his thoughts  on hammock localization and operads.

\section{Operads} \label{sec:operads}
Let $(\cD,\otimes , U)$ be a closed symmetric monoidal category 
that is tensored over sets and has (finite) colimits.  
Our primary interest is the case where $\cD$ is the category of simplicial sets $\sset$ or the category of based simplicial sets $\sset_*$.  

\begin{defn}
An {\bf operad} $\op$ in $\cD$ is a collection of objects 
  $ \{ \op(n)\}_{n\geq 0}$ in $\cD$, a map $\unit\colon U\to \op(1)$, 
  a right action of the symmetric group $\Sigma_n$ on $\op(n)$  for each $n\geq 0$, and structure maps
  \[ 
   \gamma \colon  \op(k) \otimes  \op({j_1}) \otimes \ldots\otimes \op({j_k}) \to  \op(j)
  \]
  for $k\geq 1$, $j_s \geq 0$ and $j= \Sigma_{s=1}^k j_s$ so the following diagrams commute for all  $i_t$, $j_s$ and $k$.
  \begin{itemize}
  \item $\gamma$ is associative.
  \begin{center}
\begin{tikzpicture}
 \node [right] at (0,0)(a) {$\begin{aligned}\op(k) &\otimes \op(j_1)\otimes \ldots \otimes \op(j_k)
 \\
 &\otimes 
 \op(i_1)\otimes \ldots \otimes \op(i_{j_1}) \\ &\otimes \cdots \\
 & \otimes \op (i_{j-j_k+1}) \ldots \otimes \op(i_j) \end{aligned}$};
  \node [left] at (12,0)(b) {$\begin{aligned}&\op(j_1+ \ldots+j_k) \\&\otimes \op(i_1)\otimes \ldots \otimes \op(i_{j_1}) \\
  & \otimes \ldots \\
  &\otimes \op(i_{j-j_{k}+1}) \otimes \ldots \otimes \op (i_{j})\end{aligned}$};
    \node [right] at (0.2,-4)(d) {$\begin{aligned}\op(k)&\otimes \op(i_1+\ldots+i_{j_1} )\\ &\otimes \cdots \\ &\otimes \op(i_{j-j_k+1}+ \ldots +i_j)\end{aligned}$};
    \node [left] at (11.1,-4)(c) {$ \op(i_1+\ldots +i_{j})$};
    
    \draw (a) edge [->](b);
    \draw (a) edge [->](d);
    \draw (b) edge [->](c);
    \draw (d) edge [->] (c);
\end{tikzpicture}
\end{center}
\item $\unit$ is a unit for $\gamma$.
  \[\xymatrix{
  U\otimes \op(j)\ar[d]_{\unit\otimes \id}\ar@/^/[dr]^\sim
  & 
  &
    \op(k)\otimes U\otimes \ldots \otimes U\ar[d]_{\id\otimes\unit\otimes \ldots \otimes \unit}\ar@/^/[dr]^\sim
  & 
  \\
  \op(1)\otimes \op(j)\ar[r]^-\gamma &\op(j)
  &\op(k)\otimes \op(1)\otimes \ldots \otimes \op(1)\ar[r]^-\gamma &\op(k)
  }\]
  \item $\gamma$ is equivariant with respect to the symmetric group actions.  
\[
  \xymatrix{\op(k) \otimes \op(j_1) \otimes \cdots \otimes \op(j_k) \ar[d]^{\id \otimes \sigma^\ast} \ar[r]^{\sigma \otimes \id \otimes \cdots \otimes \id} 
  & \op(k) \otimes \op(j_1) \otimes \cdots \otimes \op(j_k) \ar[dd]^{\gamma} 
  \\ 
  \op(k) \otimes \op(j_{\sigma^{-1}(1)}) \otimes \cdots \otimes \op(j_{\sigma^{-1}(k)}) \ar[d]^{\gamma} 
  \\ 
  \op(j_{\sigma^{-1}(1)}+ \cdots + j_{\sigma^{-1}(k)}) \ar[r]^{\sigma_\ast(j_1, \dots, j_k)} 
  & \op(j_1 + \cdots + j_k)
  }
  \]
  For $\sigma \in \Sigma_k$, $\sigma(j_1,\ldots,j_k)\in \Sigma_j$ permutes blocks of 
    size $j_s$ according to $\sigma$.
  \[
  \xymatrix@C=45pt{\op(k) \otimes \op(j_1) \otimes \cdots \otimes \op(j_k) \ar[r]^{\id \otimes \tau_1 \otimes \cdots \otimes \tau_k}  \ar[d]^{\gamma} 
  & \op(k) \otimes \op(j_1) \otimes \cdots \otimes \op(j_k) \ar[d]^{\gamma} 
  \\ 
  \op(j_{1}+ \cdots + j_{k}) \ar[r]^{\tau_1\oplus \dots\oplus \tau_k} 
  & \op(j_1 + \cdots + j_k)
  }
  \] 
  For $\tau_s \in \Sigma_{j_s}$ $\tau_1 \oplus \ldots \oplus \tau_k$ denotes the image of
    $(\tau_1,\ldots ,\tau_k)$ under the natural inclusion of 
    $\Sigma_{j_1}\times \ldots \times \Sigma_{j_k}$ into $\Sigma_j$
 \end{itemize}
\end{defn}

\begin{eg}\label{eg:monoid}
  Let $M$ be a simplicial monoid. Then $M$ defines an operad $M_+$ where 
  \[
   M_+(0) = \{ *\}, \,\,  M_+(1)=M, \text { and } M_+ (n) = \emptyset  \text { for } n\geq 2.
  \]
  Operad composition is defined by monoid multiplication.
  Similarly, for any operad $\op$ the monoid of one-ary operations gives rise to a suboperad $\op(1)_+$.
\end{eg}

\begin{eg}\label{eg:moduli}
  Let $\cM_{g,n}$ denote the moduli spaces of Riemann surfaces of genus $g$ with $n$ parametrized and ordered boundary components. 
  Segal \cite{Segal} constructed a symmetric monoidal category where the objects are finite unions of circles 
  and morphism spaces are  disjoint unions of the spaces $\cM_{g,n}$ with boundary circles divided into incoming and outgoing.   
  Composition of morphisms is defined by gluing outgoing circles of one Riemann surface to incoming circles of another.  
  A conformal field theory in the sense of \cite{Segal} is then a symmetric monoidal functor from this surface category to an appropriate linear category.
 
  By restricting the category and replacing the spaces by their total singular complexes, we define an operad  $\cM$ where 
  \[
   \cM(n)\coloneqq \coprod_{g\geq 0} \mathcal Sing_* (\mathcal M_{g, n+1}). 
  \] 
  In order to study stable phenomena, one wants to invert the action of the torus $T \in 
 \mathcal Sing_0(\mathcal M_{1, 1+1} ) \subset \mathcal M(1) $. See, for example, \cite{freedhopkins,TillmannComm}.  
\end{eg}

The structure and unit maps define a map $\circ_i\colon \op(k)\otimes \op(j)\to \op(k+j-1)$  by requiring the following diagram to commute.   
\[\xymatrix{
\op(k)\otimes U\otimes \ldots \otimes U\otimes \op(j)\otimes U \otimes \ldots \otimes U 
\ar[d]^{\id\otimes \unit \otimes \ldots \otimes \unit \otimes \id \otimes \unit \otimes \ldots \otimes \unit}\ar[r]^-\sim
&\op(k)\otimes \op(j)
\ar@{.>}[d]^{\circ_i}
\\
\op(k)\otimes\op(1)\otimes  \ldots \otimes \op(1)\otimes \op(j)\otimes \op(1) \otimes \ldots \otimes \op(1)\ar[r]^-\gamma
&\op(j+k-1)
}\] 
In the top left corner of the diagram $\op(j)$ is in the $(i+1)^{st}$ spot. 
If $\otimes$ is the (categorical) product in $\cD$, as it is for sets or unbased spaces,  
the maps $\circ_i$ determine the composition maps $\gamma$.

\section{Operads as symmetric monoidal categories}\label{sec:smc}
It is well-known that operads give rise to symmetric monoidal categories. 
 As the correspondence provides us with a useful  comparison (see Section \ref{sec:loc_operads}), 
we include a full description and proof of this correspondence.

For  an operad $\op$ in $\cD$,  
we define a  strict  symmetric monoidal category $\cC_\op$ enriched in $\cD$ whose objects are the natural numbers. 
Morphisms for $a>0$ are given by 
\[
 {\cC_\op}(a,b)\coloneqq\coprod_{\Sigma k_i = b} (\op(k_1) \otimes \cdots \otimes \op(k_a))\times_{\Sigma _{k_1} \times \dots \times \Sigma _{k_a}} {\Sigma_b}
\]
where the coproduct is indexed by sequences of natural numbers $(k_1, \dots , k_a)$ such that $\Sigma_i k_i = b$ and $k_i \geq 0$.
Tensoring with $\Sigma_b$ and using a coequalizer allows for all 
possible permutations of inputs, not only the ones that permute the inputs of each $\op(k_i)$ separately. 
In particular,  there is a monoid homomorphism $U\times \Sigma_n\to \cC_\op (n, n)$
and this defines a left $\Sigma_a$ and a right $\Sigma_b$ action on $\cC_\op(a,b)$.  See \autoref{fig:tree_action} for an example in a concrete category.
We define $\cC_\op(0,0)$ to be $U$ and note that $\cC_\op(0,b)$ is the initial object if $b>0$ and $\cC_\op(a,0)$ is $\op(0)^{\otimes a}$ for $a>0$.

 \begin{figure}
 \begin{subfigure}[t]{.45\linewidth}
 \centering
        \resizebox{\linewidth}{!}{
\begin{tikzpicture}
  \node at (2,1)(e1){$1$};
  \node at (2,3)(e3){$2$};
  \node at (2,5.5)(e5){$3$};

  \node at (3,1)(f1){$2$};
  \node at (3,3)(f3){$1$};
  \node at (3,5.5)(f5){$3$};

  \node at (4,1)(cc2){$1$};
  \node at (4,3)(cc3){$2$};
  \node at (4,5.5)(cc4){$3$};

  \coordinate (c1) at (5,0);
  \coordinate (c2) at (5,1);
  \coordinate (c3) at (5,3);
  \coordinate (c4) at (5,5.5);
  \coordinate (c5) at (5,7);

  \node at (6,1)(d2){$1$};
  \node at (6,2)(d3){$1$};
  \node at (6,3)(d4){$2$};
  \node at (6,4)(d5){$3$};

  \node at (6,5) (d6){$1$};
  \node at (6,6) (d7){$2$};
  
  \node at (7,1)(g2){$1$};
  \node at (7,2)(g3){$2$};
  \node at (7,3)(g4){$3$};
  \node at (7,4)(g5){$4$};

  \node at (7,5) (g6){$5$};
  \node at (7,6) (g7){$6$};
  
  \node at (8,1)(h2){$2$};
  \node at (8,2)(h3){$1$};
  \node at (8,3)(h4){$4$};
  \node at (8,4)(h5){$5$};

  \node at (8,5) (h6){$3$};
  \node at (8,6) (h7){$6$};
  \draw (cc2)--(c2) -- (d2);
  \draw (cc3)--(c3) -- (d3);
  \draw (c3) -- (d4);
  \draw (c3) -- (d5);
  \draw (cc4)--(c4)-- (d6);
  \draw (c4)-- (d7);
  
   \draw [dotted] (e5)--(f5);
   \draw [dotted] (e3)--(f3);
   \draw [dotted] (e1)--(f1);

  \draw [dotted] (d2)--(g2)--(h2);
  \draw [dotted] (d3)--(g3)--(h3);
  \draw [dotted] (d4)--(g4)--(h4);
  \draw [dotted] (d5)--(g5)--(h5);
  \draw [dotted] (d6)--(g6)--(h6);
  \draw [dotted] (d7)--(g7)--(h7);
\end{tikzpicture}}
\caption{$\sigma=(12)$, $\tau=(12)(345)$}
\end{subfigure}
\hspace{.5cm}
\begin{subfigure}[t]{.45\linewidth}
 \centering
        \resizebox{\linewidth}{!}{

\begin{tikzpicture}

  \node at (2,2)(e1){$1$};
  \node at (2,4)(e3){$2$};
  \node at (2,5.5)(e5){$3$};

  \node at (3,2)(f1){$2$};
  \node at (3,4)(f3){$1$};
  \node at (3,5.5)(f5){$3$};
  
  \node at (4,4)(cc2){$1$};
  \node at (4,2)(cc3){$2$};
  \node at (4,5.5)(cc4){$3$};
  
  \coordinate (c2) at (5,4);
  \coordinate (c3) at (5,2);
  \coordinate (c4) at (5,5.5);
  
  \node at (6,4)(d2){$1$};
  \node at (6,1)(d3){$1$};
  \node at (6,2)(d4){$2$};
  \node at (6,3)(d5){$3$};

  \node at (6,5) (d6){$1$};
  \node at (6,6) (d7){$2$};
  
  \node at (7,4)(g2){$4$};
  \node at (7,1)(g3){$1$};
  \node at (7,2)(g4){$2$};
  \node at (7,3)(g5){$3$};
  \node at (7,5) (g6){$5$};
  \node at (7,6) (g7){$6$};

  \node at (8,4)(h2){$2$};
  \node at (8,1)(h3){$1$};
  \node at (8,2)(h4){$4$};
  \node at (8,3)(h5){$5$};

  \node at (8,5) (h6){$3$};
  \node at (8,6) (h7){$6$};
  
  \draw (cc2)--(c2) -- (d2);
  \draw (cc3)--(c3) -- (d3);
  \draw (c3) -- (d4);
  \draw (c3) -- (d5);
  \draw (cc4)--(c4)-- (d6);
  \draw (c4)-- (d7);

   \draw [dotted] (e5)--(f5)--(cc4);
   \draw [dotted] (e3)--(f3)--(cc2);
   \draw [dotted] (e1)--(f1)--(cc3);

  \draw [dotted] (d2)--(g2)--(h2);
  \draw [dotted] (d3)--(g3)--(h3);
  \draw [dotted] (d4)--(g4)--(h4);
  \draw [dotted] (d5)--(g5)--(h5);
  \draw [dotted] (d6)--(g6)--(h6);
  \draw [dotted] (d7)--(g7)--(h7);
\end{tikzpicture}}
\caption{$\sigma(1,3,2)\tau=(24)(35)$}
\end{subfigure}
\caption{The left symmetric group action}\label{fig:tree_action}
\end{figure}
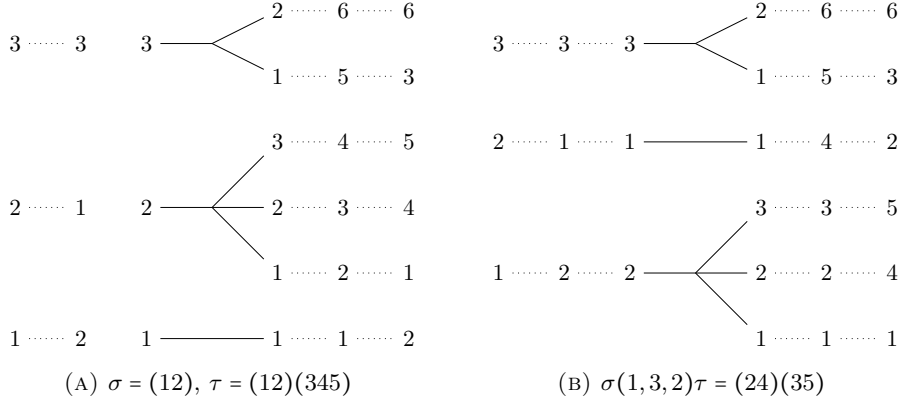

\vspace{3pt}
\noindent{\bf Composition.} The map $\cC_\op(a,b)\otimes\cC_\op(b,c) \to \cC_\op(a,c)$ is factored by the map 
\[
\xymatrix{
\left(\op(k_1) \otimes \cdots \otimes \op(k_a) \times_{\Sigma_{k_1} \times \cdots \times \Sigma_{k_a}} \Sigma_b \right) \otimes \left(\op(j_1) \otimes \cdots \otimes \op(j_b) \times_{\Sigma_{j_1} \times \cdots \times \Sigma_{j_b}} \Sigma_c \right)
\ar[d] \\
\left( \op(k_1) \otimes \cdots \otimes \op(k_a) \otimes \op(j_1) \otimes \cdots \otimes \op(j_b) \right) \times_{\Sigma_{k_1} \times \cdots \times \Sigma_{j_b}} \Sigma_b \times \Sigma_c}
\]
The composition takes $(\sigma,\tau) \in \Sigma_b \times \Sigma_c$ to $\sigma(k_1,\ldots ,k_a)\tau$ (the block permutation described above) and is defined by $(\gamma \otimes \cdots \otimes \gamma)\circ \sigma^\ast$ on the tensor product.
 \[ 
\xymatrix{
\op(k_1) \otimes \cdots \otimes \op(k_a)  \otimes \op(j_1) \otimes \cdots \otimes \op(j_b)  \ar[d]^{\sigma^\ast} \\ 
\op(k_1) \otimes \op(j_{\sigma^{-1}(1)}) \otimes \cdots \otimes \op(j_{\sigma^{-1}(k_1)}) \otimes \cdots \otimes \op(k_a) \otimes \op(...)  \ar[d]^{\gamma \otimes \cdots \otimes \gamma} \\ 
\op(j_{\sigma^{-1}(1)} + \cdots + j_{\sigma^{-1}(k_1)}) \otimes \cdots \otimes \op(j_{\sigma^{-1}(b-k_a+1)}+ \cdots + j_{\sigma^{-1}(b)})
} \]

 \begin{figure} 
 \centering
 \begin{subfigure}[t]{.75\linewidth}
 \centering
        \resizebox{\linewidth}{!}{
\begin{tikzpicture}
  \coordinate (a1) at (0,2);
  \coordinate (a2) at (0,4.5);
  
  \node at (-1,2)(aa1){$1$};
  \node at (-1,4.5)(aa2){$2$};
  
  \node at (1,1.5)(b1){$1$};
  \node at (1,2.5)(b2){$2$};
  \node at (1,3.5)(b3){$1$};
  \node at (1,4.5)(b4){$2$};
  \node at (1,5.5)(b5){$3$};
    
  \node at (2,1.5)(e1){$1$};
  \node at (2,2.5)(e2){$2$};
  \node at (2,3.5)(e3){$3$};
  \node at (2,4.5)(e4){$4$};
  \node at (2,5.5)(e5){$5$};

  \node at (3,1.5)(f1){$3$};
  \node at (3,2.5)(f2){$1$};
  \node at (3,3.5)(f3){$2$};
  \node at (3,4.5)(f4){$5$};
  \node at (3,5.5)(f5){$4$};

  \node at (4,0)(cc1){$1$};
  \node at (4,1)(cc2){$2$};
  \node at (4,3)(cc3){$3$};
  \node at (4,5.5)(cc4){$4$};
  \node at (4,7)(cc5){$5$};

  \coordinate (c1) at (5,0);
  \coordinate (c2) at (5,1);
  \coordinate (c3) at (5,3);
  \coordinate (c4) at (5,5.5);
  \coordinate (c5) at (5,7);

  \node at (6,0)(d1){$1$};
  \node at (6,1)(d2){$1$};
  \node at (6,2)(d3){$1$};
  \node at (6,3)(d4){$2$};
  \node at (6,4)(d5){$3$};

  \node at (6,5) (d6){$1$};
  \node at (6,6) (d7){$2$};
  \node at (6,7) (d8){ $1$};
  
  \node at (7,0)(g1){$1$};
  \node at (7,1)(g2){$2$};
  \node at (7,2)(g3){$3$};
  \node at (7,3)(g4){$4$};
  \node at (7,4)(g5){$5$};

  \node at (7,5) (g6){$6$};
  \node at (7,6) (g7){$7$};
  \node at (7,7) (g8){ $8$};

  \draw (aa1) -- (a1) -- (b1);
  \draw (a1) -- (b2);
  \draw (aa2) -- (a2) -- (b3);
  \draw (a2) -- (b4);
  \draw (a2) -- (b5);

  \draw (cc1)--(c1) -- (d1);
  \draw (cc2)--(c2) -- (d2);
  \draw (cc3)--(c3) -- (d3);
  \draw (c3) -- (d4);
  \draw (c3) -- (d5);
  \draw (cc4)--(c4)-- (d6);
  \draw (c4)-- (d7);
  \draw (cc5)--(c5) -- (d8);

  \draw [dotted] (b5)--(e5)--(f5)-- (cc4);
  \draw [dotted] (b4)--(e4)--(f4)-- (cc5);
  \draw [dotted] (b3)--(e3)--(f3)-- (cc2);
  \draw [dotted] (b2)--(e2)--(f2)-- (cc1);
  \draw [dotted] (b1)--(e1)--(f1)-- (cc3);

  \draw [dotted] (d1)--(g1);
  \draw [dotted] (d2)--(g2);
  \draw [dotted] (d3)--(g3);
  \draw [dotted] (d4)--(g4);
  \draw [dotted] (d5)--(g5);
  \draw [dotted] (d6)--(g6);
  \draw [dotted] (d7)--(g7);
  \draw [dotted] (d8)--(g8);

\end{tikzpicture}}
\caption{Morphisms in $\cC_\op(2,5)$ and $\cC_\op(5,8)$.}
 \end{subfigure}
 \begin{subfigure}[t]{.45\linewidth}
 \centering
        \resizebox{\linewidth}{!}{

\begin{tikzpicture}
  \coordinate (a1) at (0,2);
  \coordinate (a2) at (0,4.5);
  
  \node at (-1,2)(aa1){$1$};
  \node at (-1,4.5)(aa2){$2$};
  
  \node at (1,1.5)(b1){$1$};
  \node at (1,2.5)(b2){$2$};
  \node at (1,3.5)(b3){$3$};
  \node at (1,4.5)(b4){$4$};
  \node at (1,5.5)(b5){$5$};

  \node at (2.5,1.5)(f1){$3$};
  \node at (2.5,2.5)(f2){$1$};
  \node at (2.5,3.5)(f3){$2$};
  \node at (2.5,4.5)(f4){$5$};
  \node at (2.5,5.5)(f5){$4$};

  \node at (4,3)(cc1){$1$};
  \node at (4,4)(cc2){$2$};
  \node at (4,1)(cc3){$3$};
  \node at (4,6.5)(cc4){$4$};
  \node at (4,5)(cc5){$5$};

  \coordinate (c1) at (5,3);
  \coordinate (c2) at (5,4);
  \coordinate (c3) at (5,1);
  \coordinate (c4) at (5,6.5);
  \coordinate (c5) at (5,5);

  \node at (6,3)(d1){$1$};
  \node at (6,4)(d2){$2$};
  \node at (6,0)(d3){$3$};
  \node at (6,1)(d4){$4$};
  \node at (6,2)(d5){$5$};

  \node at (6,6) (d6){$6$};
  \node at (6,7) (d7){$7$};
  \node at (6,5) (d8){ $8$};

  \draw (aa1) -- (a1) -- (b1);
  \draw (a1) -- (b2);
  \draw (aa2) -- (a2) -- (b3);
  \draw (a2) -- (b4);
  \draw (a2) -- (b5);

  \draw (cc1)--(c1) -- (d1);
  \draw (cc2)--(c2) -- (d2);
  \draw (cc3)--(c3) -- (d3);
  \draw (c3) -- (d4);
  \draw (c3) -- (d5);
  \draw (cc4)--(c4)-- (d6);
  \draw (c4)-- (d7);
  \draw (cc5)--(c5) -- (d8);

  \draw [dotted] (b5)--(f5)-- (cc4);
  \draw [dotted] (b4)--(f4)-- (cc5);
  \draw [dotted] (b3)--(f3)-- (cc2);
  \draw [dotted] (b2)--(f2)-- (cc1);
  \draw [dotted] (b1)--(f1)-- (cc3);

\end{tikzpicture}}
\caption{Acting by the symmetric group.}
\end{subfigure}
 \begin{subfigure}[t]{.45\linewidth}
 \centering
        \resizebox{\linewidth}{!}{

\begin{tikzpicture}
  \coordinate (a1) at (0,2);
  \coordinate (a2) at (0,4.5);
  
  \node at (-1,2)(aa1){$1$};
  \node at (-1,4.5)(aa2){$2$};
  
  \coordinate (b1) at  (1,1.5);
  \coordinate (b2) at   (1,2.5);
  \coordinate (b3) at   (1,3.5);
  \coordinate (b4) at   (1,4.5);
  \coordinate (b5) at   (1,5.5);

  \coordinate (f1) at   (1.5,1.5);
  \coordinate (f2) at   (1.5,2.5);
  \coordinate (f3) at   (1.5,3.5);
  \coordinate (f4) at   (1.5,4.5);
  \coordinate (f5) at   (1.5,5.5);

  \coordinate (cc1) at   (2,3);
  \coordinate (cc2) at   (2,4);
  \coordinate (cc3) at   (2,1);
  \coordinate (cc4) at   (2,6.5);
  \coordinate (cc5) at   (2,5);

  \coordinate (c1) at (3,3);
  \coordinate (c2) at (3,4);
  \coordinate (c3) at (3,1);
  \coordinate (c4) at (3,6.5);
  \coordinate (c5) at (3,5);

  \node at (4,3)(d1){$4$};
  \node at (4,4)(d2){$1$};
  \node at (4,0)(d3){$1$};
  \node at (4,1)(d4){$2$};
  \node at (4,2)(d5){$3$};

  \node at (4,6) (d6){$3$};
  \node at (4,7) (d7){$4$};
  \node at (4,5) (d8){ $2$};

  \node at (5,3)(e1){$4$};
  \node at (5,4)(e2){$5$};
  \node at (5,0)(e3){$1$};
  \node at (5,1)(e4){$2$};
  \node at (5,2)(e5){$3$};

  \node at (5,6) (e6){$7$};
  \node at (5,7) (e7){$8$};
  \node at (5,5) (e8){ $6$};
  
    \node at (6,3)(g1){$1$};
  \node at (6,4)(g2){$2$};
  \node at (6,0)(g3){$3$};
  \node at (6,1)(g4){$4$};
  \node at (6,2)(g5){$5$};

  \node at (6,6) (g6){$6$};
  \node at (6,7) (g7){$7$};
  \node at (6,5) (g8){ $8$};

  \draw (aa1) -- (a1) -- (b1);
  \draw (a1) -- (b2);
  \draw (aa2) -- (a2) -- (b3);
  \draw (a2) -- (b4);
  \draw (a2) -- (b5);

  \draw (cc1)--(c1) -- (d1);
  \draw (cc2)--(c2) -- (d2);
  \draw (cc3)--(c3) -- (d3);
  \draw (c3) -- (d4);
  \draw (c3) -- (d5);
  \draw (cc4)--(c4)-- (d6);
  \draw (c4)-- (d7);
  \draw (cc5)--(c5) -- (d8);

  \draw (b5)--(f5)-- (cc4);
  \draw (b4)--(f4)-- (cc5);
  \draw  (b3)--(f3)-- (cc2);
  \draw (b2)--(f2)-- (cc1);
  \draw  (b1)--(f1)-- (cc3);

 \draw [dotted] (d1)--(e1)--(g1);
 \draw [dotted] (d2)--(e2)--(g2);
 \draw [dotted] (d3)--(e3)--(g3);
 \draw [dotted] (d4)--(e4)--(g4);
 \draw [dotted] (d5)--(e5)--(g5);
 \draw [dotted] (d6)--(e6)--(g6);
 \draw [dotted] (d7)--(e7)--(g7);
 \draw [dotted] (d8)--(e8)--(g8);

\end{tikzpicture}}
\caption{The final composite}
\end{subfigure}
\caption{An example of composition in the category $\cC_\op$.}\label{fig:tree_composition}
\end{figure}

\vspace{3pt}
\noindent{\bf Monoidal structure.}  
The addition of natural numbers defines the monoidal structure on the objects of $\cC_\op.$ On morphisms
the monoidal product 
\[
  \cC_\op(a,b)\otimes \cC_{\op}(c,d)\to  \cC_\op(a+c,b+d)
\]
is the map
\begin{align*}
  (\op(k_1) &\otimes \cdots \otimes \op(k_a))\times_{\Sigma _{k_1} \times \dots \times \Sigma _{k_a}} {\Sigma_b}
  \otimes (\op(j_1) \otimes \cdots \otimes \op(j_c))\times_{\Sigma _{j_1} \times \dots \times \Sigma _{j_c}} {\Sigma_d}
  \\
  &\to 
  \left(\op(k_1) \otimes \cdots \otimes \op(k_a)
  \otimes \op(j_1) \otimes \cdots \otimes \op(j_c)\right)\times_{\Sigma_{k_1}\times
  \ldots   \Sigma_{k_a}\times \Sigma _{j_1} \times \dots \times \Sigma _{j_c}} {\Sigma_{b+d}}
\end{align*}
induced by the canonical inclusion 
$\Sigma_b \times \Sigma_d \to \Sigma _{b+d}$.
Note that $\cC_\op(a,b)\otimes \cC_\op(c,d)$ has a left action by $\Sigma_a\times \Sigma_c$ and a right action by $\Sigma_b\times \Sigma_d$. 

If $(a,b)\in\Sigma_{b+a}$ is the permutation that permutes the first $a$ entries to the end, 
\[(U\otimes \ldots \otimes U)\times *\xto{\unit\otimes \ldots \otimes \unit\otimes (a,b)}  \op(1)\otimes \ldots
\otimes \op(1)\times_{\Sigma_1\times \ldots \times \Sigma_1}\Sigma_{b+a}\subset \cC_\op(a+b,b+a)\]
defines an isomorphism from $a+b$ to $b+a$.  
Compatibility with block permutation makes this a natural transformation.

This construction extends to an equivalence of categories.

\begin{prop}\label{smcoperad}
  There is an equivalence of categories between the category of operads in $\cD$ and 
  the category of strict symmetric monoidal categories enriched in $\cD$ such that
  \begin{enumerate}
    \item\label{item:monoid_objects}  the monoid of objects is $(\mathbb N , +, 0)$;
    \item\label{item:equivariant}  
    {there is a monoid homomorphism $\Sigma_b\to \cC(b,b)$ for all objects $b$;}
    \item\label{item:maps} the map
    \[
      \coprod_{(k_1, \dots, k_a), \Sigma k_i = b} (\cC(1,k_1) \otimes \cdots \otimes \cC(1,k_a))
      \times_{\Sigma _{k_1} \times \dots \times \Sigma _{k_a}} {\Sigma_b} \to \cC(a,b)
    \]
    defined by the monoidal structure and composition is an isomorphism  for all $a\geq 1$ and $b$; 
    the left action by $\Sigma_a \subset \cC (a, a)$ corresponds to permuting the factors in the tensor product;  
    and  the identity map is the only map with source $0$.
  \end{enumerate}
\end{prop}

\begin{proof}
  Note that $\cC_\op$ satisfies the conditions above if $\op$ is an operad.

 For the other implication, if $\cC$ is a symmetric monoidal category satisfying the hypotheses above, define an operad $\op_\cC$ by 
  \[
    \op_\cC(n)\coloneqq \cC(1,n).
  \]  
  Then $\op_\cC(n)$ has a right action by $\Sigma_n \subset \cC (n, n)$.  The  operad structure maps 
  \[
    \op_\cC(k)\otimes (\op_\cC(j_1)\otimes \ldots \otimes \op_\cC(j_k))\to \op_\cC(j_1+\ldots +j_k)
  \]
  are the composites
  \[
    \cC(1,k)\otimes (\cC(1,j_1)\otimes \ldots \otimes \cC(1,j_k))\subset \cC(1,k)\otimes \cC(k, j_1+\ldots +j_k)\to \cC(1,j_1+\ldots +j_k).
  \]
  By assumption \ref{item:maps} this map is equivariant with respect to $\Sigma_k$ and $\Sigma_{ j_1}\times \ldots \times \Sigma_{ j_k}$.

  For an operad $\op$, 
  \[
    \op_{\cC_\op}(b)\coloneqq \cC_\op(1,b)=\op(b)\times_{\Sigma_b}\Sigma_b=\op(b).
  \]
  Starting with a symmetric monoidal category $\cC$, $\op_\cC(k)\coloneqq \cC(1,k)$ and 
  \begin{align*}
    \cC_{\op_\cC}(a,b)&
    \coloneqq \coprod_{\Sigma k_i = b} (\op_\cC(k_1) \otimes \cdots \otimes \op_\cC(k_a))\times_{\Sigma _{k_1} \times \dots \times \Sigma _{k_a}} {\Sigma_b}
    \\
    &= \coprod_{\Sigma k_i = b} (\cC(1,k_1) \otimes \cdots \otimes \cC(1,k_a))\times_{\Sigma _{k_1} \times \dots \times \Sigma _{k_a}} {\Sigma_b}
  \end{align*}
  When $\cC$ satisfies the third condition above, this agrees with $\cC(a,b)$.
\end{proof}

\section{Hammock Localization for categories enriched in simplicial sets}\label{sec:hammock}

We recall the hammock localization of \cite{DwyerKan} and its extension to simplicially enriched categories \cite[2.5]{DwyerKan}. 

For a category $\cC$ and a subcategory $\cW$, a {\bf reduced hammock} of height $n$ in $\cC$ is a commutative diagram in  $\cC$ (of arbitrary length $m\geq 0$) of the form
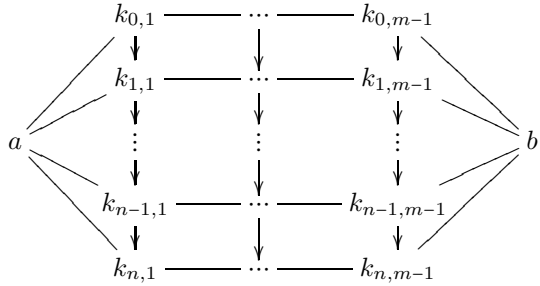
\[
 \xymatrix@R=2ex{
	& k_{0,1} \ar[d] \ar@{-}[r] & \cdots \ar@{-}[r] \ar[d] & k_{0,m-1} \ar[d] \ar@{-}[ddr] 
    \\
    & k_{1,1} \ar[d] \ar@{-}[r] & \cdots \ar@{-}[r] \ar[d] & k_{1,m-1} \ar[d]  \ar@{-}[dr] 
    \\ 
    a \ar@{-}[ur] \ar@{-}[uur] \ar@{-}[dr]  \ar@{-}[ddr] & \vdots \ar[d] & \vdots \ar[d] & \vdots \ar[d] & b 
    \\
    & k_{n-1,1} \ar@{-}[r] \ar[d]& \cdots \ar@{-}[r]\ar[d] & k_{n-1,m-1} \ar@{-}[ur]\ar[d]
    \\
    & k_{n,1} \ar@{-}[r] & \cdots \ar@{-}[r] & k_{n,m-1} \ar@{-}[uur]
 }
\]
satisfying the following conditions:
\begin{itemize}
 \item horizontal arrows can point either direction,
 \item vertical and left-pointing arrows are in $\cW$,
 \item in each column the horizontal arrows point in the same direction, 
 \item arrows in adjacent columns point in different directions, and 
 \item no column contains only identity maps. 
\end{itemize}
The {\bf hammock localization} of $\cC$ at $\cW$, $L^H_\cW\cC$, is a simplicially enriched category whose objects are the objects of $\cC$.  
The $n$-simplices in ${L^H_\cW\cC}(a,b)$ are the reduced hammocks of height $n$.  
We define face and degeneracy maps by deleting or repeating a row, then reducing according to the conditions above.
Composition in the category $L^H_\cW\cC$ is via concatenation of hammocks where the point of gluing is expanded to a vertical column of identity maps.

If  $O$ is a fixed set then $O-\cat$ is the category of small categories whose object set is $O$.  
Similarly, $sO-\cat$ is the categories enriched in $\sset$ whose object set is $O$. 

\begin{thm}\cite[3.4]{DwyerKan}
  Localization is a functor 
  \[\{(\cC,\cW)|\cW\subset \cC\in O-\cat \} \to sO-\cat\]
  from the category of $O$-categories with an $O$-subcategory to the category of simplicial $O$-categories.
  For each pair  $(\cC, \cW)$, there are canonical inclusion functors $\cC\to L^H_\cW\cC$ and $\cW^{op} \to L^H_\cW\cC$.
\end{thm}

The map $\cC\to L^H_\cW\cC$ is given by the inclusion of hammocks of length one  and height zero with arrows pointing to the right. 
Similarly, the functor $\cW^{op} \to L^H_\cW\cC$ is given by the inclusion of hammocks of length one  and height zero  with arrows pointing to the left.

We now generalize to simplicially enriched categories.  
Let $\cC_*$ be a simplicially enriched $O$-category and $\cW_*$ be an $O$-subcategory of $\cC_*$.  
For each simplicial degree $k$, let $\cC_k$ be the $O$-category whose morphisms are the $k$ simplices of  ${\cC}_*(a,b).$  
The face and degeneracy maps in the hom sets of $\cC$ define functors $\cC_k\to \cC_{k-1}$ and $\cC_k\to \cC_{k+1}$.  
We also have corresponding subcategories $\cW_k$, which form a simplicial subcategory.

The hammock localization of $\cC_k$ with respect to $\cW_k$ is a simplicially enriched category $L^H_{\cW_k}\cC_k$.   
The {\bf hammock localization} of $\cC_*$ at $\cW_*$ is a category $L^H_{\cW_*}\cC_\ast$, enriched in bisimplicial sets, 
where $L^H_{\cW_*}\cC_\ast(a,b)_{k,\ell}$ is the height $\ell$ hammocks of $L^H_{\cW_k}\cC_k(a,b)$.  
The required simplicial structure maps are defined using  the functoriality of the hammock localization.

\begin{prop}\label{loc_is_functor}
  Let $\cC_*$ and $\cC_*'$ be simplicially enriched $O$-categories with $O$-subcategories $\cW_*\subset \cC_*$ and 
  $\cW_*'\subset \cC_*'$.  A functor $F\colon \cC_*\to \cC_*'$ 
  of simplicially enriched categories such that $F(\cW_*)\subset \cW_*'$ induces a functor of bisimplicially enriched categories
  \[
    L^H_{\cW_*}\cC_*\to L^H_{\cW_*'}\cC'_*.
  \]
\end{prop}

\begin{proof}
  By the corresponding result for unenriched categories \cite[3.4]{DwyerKan}, there are maps 
  $L^H_{\cW_k}\cC_k\to L^H_{\cW_k'}\cC_k'$.   These assemble to a functor $L^H_{\cW_*}\cC_\ast\to L^H_{\cW_*'}\cC_\ast'$ 
  of categories enriched in bisimplicial sets.  Applying the diagonal functor induces a simplicial functor as in the statement. 
\end{proof}

\begin{prop}\label{loc_we}
  Let $\phi \colon b\to c$ be in $\cW_*$.   Then for every object $a$ in $\cC$, $\phi$ induces weak homotopy equivalences
  \[
    \phi_*\colon L^H_{\cW_*}\cC_*(a,b)\to L^H_{\cW_*}\cC_*(a,c) \text{  and  }
    \phi^*\colon L^H_{\cW_*}\cC_*(c,a) \to L^H_{\cW_*}\cC_*(b,a)
  \]
\end{prop}

\begin{proof}
  By \cite[3.3]{DwyerKan}, the maps $(\phi_k)_*\colon L^H_{\cW_k}\cC_k(a,b) \to L^H_{\cW_*}\cC_k(a,c)$
  and $(\phi_k)^*\colon L^H_{\cW_*}\cC_k(c,a) \to L^H_{\cW_*}\cC_k(b,a)$ 
  are weak homotopy equivalences.  Then \cite[4.1.7]{goerssjardine} implies they
  are weak homotopy equivalences.  
\end{proof}

 \section{Localization for Operads}\label{sec:loc_operads}

Ideally, we would like to consider an operad as one of the symmetric monoidal categories identified in Section \ref{sec:smc}, 
and apply the hammock localization as described in Section \ref{sec:hammock}.
However, in order to translate back to operads, we would need to show that the localized category is of the form  prescribed by \autoref{smcoperad}.
Unfortunately, there is an obstacle. In order to define the monoidal product on hammocks 
one has to struggle with the fact that different hammocks have different lengths, 
and any naive attempt to extend hammocks via identities seems to lead to 
a monoidal product that no longer commutes with composition and hence will not be functorial.

\vskip .1in
Here we instead give a construction working directly with the operad. 
The standard hammock localization applied to the monoid of 1-ary operations $\op (1)$  
is a subcategory of the localization of $\op(1)$ constructed in this section. 
Indeed,  our procedure will construct a localization of the associated strict symmetric monoidal category as described in \autoref{smcoperad}.
We expect that this localization will  yield a category that is homotopic to the standard hammock localization.

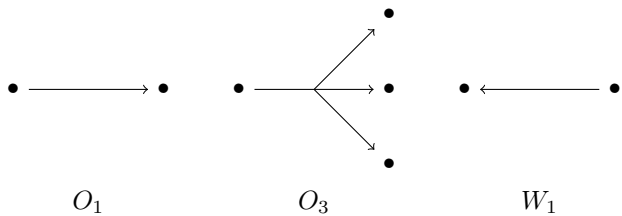
\begin{figure}
\centering
\begin{tikzpicture}
\node at (-3,0)(a01) {$\bullet$};
\node at (-1,0)(a0) {$\bullet$};
\draw [->](a01) --(a0);

\node at (-2, -1.5) {$O_1$};

\node at (0,0)(a1) {$\bullet$};
\coordinate (t1) at (1, 0);
\node at (1, -1.5) {$O_3$};
\node at (2,1)(a22) {$\bullet$};
\node at (2,0)(a23) {$\bullet$};

\node at (2,-1)(a32) {$\bullet$};

\node at (3,0)(a33) {$\bullet$};

\node at (5, 0)(a43) {$\bullet$};
\node at (4, -1.5) {$W_1$};

\draw [->](a1) -- (t1)-- (a22);
\draw [->](t1)--(a23);
\draw [->](t1)--(a32);
\draw [<-](a33) --(a43);
\end{tikzpicture}
\caption{Atomic building blocks. The root node is on the left for each tree.}\label{fig:block}
\end{figure}
\vskip .1in
The major difference between this section and the previous is that we now allow hammocks where 
the rows have been replaced by trees with one root and $n$ labeled leaves.
For $n \geq 0$, let $\mathbb{T}(n)$ be the set of directed planar trees with one root and 
$n$ leaves labeled $1$ through $n$ and possibly some other leaves without label.
Each tree $\tau \in \mathbb T(n)$ is constructed from atomic, directed pieces.
These atomic pieces are
\[
  O_n\colon 1 \to n \text { for } n\geq 0,  \text { and }  W_1\colon 1 \leftarrow 1.
 \]
{See \autoref{fig:block}.}
Each tree has a root node (the source for $O_n$ and the target for $W_1$) and leaf nodes; 
each tree is constructed by gluing together atomic pieces  identifying the root node of one tree with a labeled leaf node of another tree; 
the resulting node will be called an internal node.
\footnote{The nodes are not the vertices in the underlying graph. For example
$O_n\colon 1 \to n$ has $n+1$ nodes corresponding to its boundary but no internal one.}

\vskip .1in
Let $\op$ be an operad and let $\cW$ be a sub-monoid of operations in $\op(1)$.
A {\bf reduced tree hammock} of height $k$ and type $\tau \in \mathbb T(n)$  is a three dimensional diagram consisting 
of $k$ copies of $\tau$ arranged in parallel (horizontal) planes with additional (vertical and downward pointing) directed edges 
connecting corresponding roots, leaves, and internal nodes in consecutive copies of $\tau$. 
 Each atomic piece in $\tau$ and each vertical edge is labeled by an element in $\op$ so that the diagram commutes. 
We further require that 
\begin{enumerate}
  \item  atomic pieces $O_n\colon 1 \to n$ are labeled by elements in $\op(n)$;
  \item  atomic pieces $W_1\colon 1 \leftarrow 1$ and vertical arrows are labeled by elements in $\cW$;
  \item  arrows in adjacent columns of atomic pieces point in different directions\footnote{In particular, for an atomic piece $O_n: 1 \to n$ all $n+1$ adjacent pieces will have to be of the form $W_1: 1 \leftarrow 1$}; and
  \item  no column of atomic pieces corresponding to $O_1$ or $W_1$ contains arrows all labeled by the identity element.  
  (Note that the vertical arrows can all be identity maps.)
\end{enumerate}
Note that if $\tau$ is made up of atomic pieces $O_1\colon 1 \to 1 $ and $W_1\colon 1 \leftarrow 1$, this is a reduced hammock  in the sense of \autoref{sec:hammock} for the monoid $\op(1)$ viewed as a category.

We have removed the initial and final maps of the hammocks. See \autoref{fig:tree_hammock} for an example of this type of hammock. 
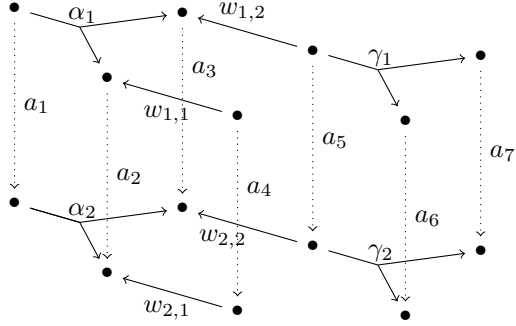
\begin{figure}
\centering
\tdplotsetmaincoords{60}{120}

\begin{tikzpicture}[tdplot_main_coords]
\node at (0,0,0)(a1) {$\bullet$};
\coordinate (t1) at (0, 1,0);
\node at (1,2,0)(a22) {$\bullet$};
\node at (-1,2,0)(a23) {$\bullet$};

\node at (1,4,0)(a32) {$\bullet$};
\node at (-1,4,0)(a33) {$\bullet$};

\coordinate (t3) at (-1, 5,0);

\node at (0,6,0)(a43) {$\bullet$};
\node at (-2,6,0)(a44) {$\bullet$};

\node at (0,0,-3)(b1) {$\bullet$};

\coordinate (t2) at (0, 1,-3);
\node at (1,2,-3)(b22) {$\bullet$};
\node at (-1,2,-3)(b23) {$\bullet$};
\node at (1,4,-3)(b32) {$\bullet$};
\node at (-1,4,-3)(b33) {$\bullet$};
\coordinate (t4) at (-1, 5,-3);

\node at (0,6,-3)(b43) {$\bullet$};
\node at (-2,6,-3)(b44) {$\bullet$};

\draw [dotted,->] (a1)--node[auto]{$a_1$}(b1);
\draw [dotted,->] (a22)--node[auto]{$a_2$}(b22);
\draw [dotted,->] (a23)--node[right, pos= 0.25]{$a_3$}(b23);
\draw [dotted,->] (a32)--node[right, pos=0.35]{$a_4$}(b32);
\draw [dotted,->] (a33)--node[right, pos= 0.45]{$a_5$}(b33);

\draw [dotted,->] (a43)--node[auto]{$a_6$}(b43);
\draw [dotted,->] (a44)--node[auto]{$a_7$}(b44);

\draw [->](a1) -- (t1)--node[above, pos=0.15] {$\alpha_1$} (a22);
\draw [<-](a22) --node[below, pos=0.45] {$w_{1,1}$}(a32);
\draw [->](t1)--(a23);
\draw [<-](a23) --node[above, pos=0.45] {$w_{1,2}$}(a33);
\draw [->](a33) --(t3)-- node[above, pos=0.15] {$\gamma_1$}(a43);
\draw [->](t3)--(a44);
\draw [->](b1)--(t2) -- node[above, pos=0.15] {$\alpha_2$}(b22);
\draw [<-](b22) --node[below, pos=0.45] {$w_{2,1}$}(b32);

\draw [->](b1)--(t2)--(b23);
\draw [<-](b23) --node[below, pos=0.25] {$w_{2,2}$}(b33);
\draw [->](b33) --(t4)-- node[above, pos=0.15] {$\gamma_2$}(b43);
\draw [->](t4)--(b44);

\end{tikzpicture}
\caption{A reduced tree hammock of height one.\\
We require 
$w_{1,1}$, $w_{1,2}$, $w_{2,1}$, $w_{2,2}$ and $ a_i$, $i=1,\ldots 7$, to be operations in $\cW$ and $\gamma(\alpha_1;a_2,a_3)=\gamma(a_1;\alpha_2)$, $\gamma(w_{1,1};a_2)=\gamma(a_4; w_{2,1})$, $\gamma(w_{1,2};a_3)=\gamma(a_5; w_{2,2})$, and $\gamma(\gamma_1;a_6,a_7)=\gamma(a_5;\gamma_2)$}\label{fig:tree_hammock}
\end{figure}

For each nonnegative integer $n$, let $L\op(n)$ be the simplicial set whose $k$ simplices are the reduced tree hammocks of height $k$ and type $\tau\in \mathbb{T}(n)$. 
Simplicial face maps are given by deleting a plane and composing adjacent vertical maps. Degeneracy maps are given by repeating a plane. 
The simplicial sets $L\op(n)$ assemble to an operad where the $i$th composition $\circ _i$ 
is defined by grafting the underlying trees, that is, identifying the leaf node labeled $i$ with the root node 
of the $i$th tree and relabeling leaf nodes as appropriate. 
If the tree hammock resulting from any of these operations is not in reduced form it can easily be made so by 
composing operations for columns of  neighboring arrows pointing in the same direction, and  by deleting columns of identities.

\vskip .1in
As for the  hammocks defined in the previous section, the construction above extends to simplicially enriched operads\footnote{  
If $\cD$ is the category of simplicial sets, we will say that such an operad in $\cD$ is a \textbf{simplicial operad} or is \textbf{simplicially enriched}.
}. 
For a simplicially enriched  operad $\op$ and a sub-monoid $\cW \subset \op(1)$, let 
$L^{TH} _\cW \op (n)$ (or $L\op(n)$) be the bisimplicial set of all reduced hammocks of tree type $\tau$ for some $\tau \in \mathbb T(n)$.  
The symmetric group acts on 
$ L^{TH} _\cW \op (n)$ 
by relabeling the labeled leaf nodes.  Grafting of tree hammocks (and reduction if necessary)  defines an associative and equivariant composition, and  we thus define the {\bf tree hammock localization} of a simplicially enriched operad $\op$ with respect to the sub-monoid $\cW \subset \op(1)$ to be the operad 
\[
  L^{TH}_\cW \op 
\]  
enriched in bisimplicial sets.

\vskip .2in
We have the following analogues of Proposition \ref{loc_is_functor} and \ref{loc_we}.

\begin{lem}[Functoriality]\label{op_loc_functor}
  Let $\op$ and $\op'$ be two simplicial operads with sub-monoids $\cW$ and $\cW'$ of the respective 1-ary operations.
  Let $\phi\colon \op \to \op'$ be a map of operads with $\phi(\cW) \subset \cW'$. Then $\phi$ induces a map on localizations
  \[
    \phi\colon L^{TH} _\cW \op \longrightarrow L^{TH} _{\cW'} \op'.
  \]
\end{lem}
 
In particular we may consider the map of pairs $(\op, \{1\}) \subset (\op, \cW)$. This defines a  natural map
\[
  \op \longrightarrow L^{TH}_{\cW} \op
\]
where $\op$ is identified with the  bisimplicial set $L^{TH}_{\{1\}}\op$ constant in one simplicial direction.

\begin{lem} [Weak invertibility] \label{op_loc_we}
  Let $w \in \cW $ be a one-ary operation. Then $w$ induces weak homotopy equivalences
  \[
    w \circ \_  \colon L^{TH}_\cW \op (n)  \to L^{TH}_\cW \op (n)   
   \,  \text { and } \,
    \_ \circ _i w\colon
    L^{TH}_\cW \op(n)  \to L^{TH}_\cW \op(n)
  \]
\end{lem}

\begin{proof}
  Composition with $w$ corresponds to grafting with a tree of type $O_1\colon 1\to 1$ and  label $w$.
  The inverse is given by  grafting with a tree of type $W_1\colon 1 \leftarrow 1$ with label $w$.
  The composition of these two operations is homotopic to the identity as can be deduced from the following hammocks of height 1
  \[
    \xymatrix{\bullet \ar[d]_w&\bullet \ar[l]_1\ar[d]^1\ar[r]^1  &\bullet \ar[d]^w\\
    \bullet  &\bullet\ar[l]^w \ar[r]_w &\bullet}
  \]
  and
  \[
    \xymatrix{\bullet \ar[r]^w\ar[d]_w&\bullet \ar[d]^1 &\bullet \ar[l]_w \ar[d]^w\\
    \bullet \ar[r]_1 &\bullet &\bullet\ar[l]^1 }
  \]
\end{proof}

\vskip .2in

We will now compare the tree hammock localization for operads with the hammock localization for categories.

\vskip .1in
By \autoref{eg:monoid}, the monoid  of one-ary operations in an operad $\op$  forms a sub-operad $M_+$ with 
\[
  M_+(1) = \op(1) , \, M_+(0) = \{ * \}, \text { and }  M_+(n) = \emptyset \text{ for  $n>1$ }.
\]
Vice versa, any monoid $M$ gives rise to such an operad $M_+$ with one-ary operations $M_+(1) =M$.
As every monoid is also a category (with one object), we thus have two potentially different ways of localizing.

\begin{lem}\label{lem:h_to_th}
  Let $M$ be a monoid and $W$ a submonoid.
  Then the  hammock localization of $M$ with respect to $W$ 
  agrees with the tree hammock localization of $M_+$ with respect to $W.$
  More precisely, there is an isomorphism of simplicial monoids
  \[
     L^{H} _W (M)  \cong L^{TH} _W (M_+) (1).
  \] 
  \qed
\end {lem}

Thus, for example, the torus $T$ from \autoref{eg:moduli} generates a (free) monoid $W \simeq \mathbb N\subset \cM (1)$. 
Its localization is homotopy equivalent to $\mathbb Z$.

\vskip .1in
More generally, let $\op$ be an operad and $\cW$ be a submonoid of $\op (1)$ which we extend to the suboperad $\cW_+$.   
By \autoref{smcoperad}, $\op$ and $\cW_+$  uniquely define a strict symmetric monoidal category $\cC_\op$ with a strict monoidal subcategory $\cC_{\cW_+}$.

\begin{prop}
  Hammock reduction defines a 
  full functor of (bi)simplicially enriched $\mathbb N$-categories
  \[
    R\colon L^H_{\cC_{\cW_+}} \cC_\op \longrightarrow \cC _{L^{TH}_{\cW}\op }.
  \]
\end{prop}

Note that the target of the reduction map $R$ is a strict symmetric monoidal category.
Thus tree hammock localization provides a localization in strict symmetric monoidal categories for categories of the form
$\cC _\op$.

\begin{proof}
By definition, $R$ is the identity on objects.

By the description of the morphisms  in $\cC_\op$ in \autoref{sec:smc},  the statement of the proposition will follow from the special case where the source is the object $1$. 

Let $n\geq 0$. 
It is not hard to see that the  reduced hammocks defining   $L^H_{\cC_{\cW_+}} \cC_\op (1,n)$ are precisely the tree hammocks  
where the geodesic paths from the root node to any of the $n$ labeled leaf nodes have exactly the same length, 
that is the same number of atomic pieces: The union of all the atomic pieces that are precisely $i$ steps away 
from the root make up the $i$th column of the standard hammock described in  \autoref{sec:hammock}. 
Note that there might be more or less than $n$ pieces $i$ steps away from the root; 
the former case may arise when there are  contributions from paths to unlabeled leaf nodes. 
Even though we start with a reduced hammock, it may not be reduced  as a tree hammock. 
This is because  hammock reduction will only remove the $i$th column in the hammock
if the individual columns of all $i$th pieces in the tree hammock are identities. On the other hand the tree hammock 
reduction removes identity columns defined by any atomic piece in the tree.
Thus every element in $L^H_{\cC_{\cW_+}} \cC _\op (1,n)$ defines an element in $L^{TH}_\cW\op (n)$  but possibly only after tree hammock reduction:
\[
R\colon L^H_{\cC_{\cW_+}}\cC _\op (1, n) \longrightarrow L^{TH}_\cW \op (n).
\]

It is easy to see that every reduced tree hammock
is the reduction of a tree hammock with equal length geodesics from root nodes to labeled leaf nodes. 
Indeed, any tree hammock can be extended to such a tree hammock by adding identity columns to the labeled nodes where necessary. Hence $R$ is surjective.

As reduction preserves the simplicial structure, $R$ defines a (bi)simplicial map. Finally, reduction also commutes with composition
(gluing of hammocks and grafting of trees). Hence $R$ is functorial. 
\end{proof}

In the special case of \autoref{lem:h_to_th}, $R$ induces  an isomorphism  on morphism spaces. In general, we believe that $R$ induces a homotopy equivalence
on morphism spaces. 
This would imply that properties of the usual hammock localization (such as homotopy invariance and calculus of fractions) can be 
transferred to the tree hammock localization of operads.
We hope to return to this question elsewhere.

\section{Localization for algebras}\label{sec:loc_alg}

We now turn to the study of algebras over $L\op = L^{TH} _\cW (\op)$.  Correspondingly we restrict our attention to operads in $\sset$. 

\vskip .1in
Recall,   a {\bf $\op$-\oalg{}} is a  based simplicial set $(X, \ast)$ with  structure maps
  \[
   \theta\colon  \op(j) \times  X^j \to  X
  \]
  for all $j \geq  0$ such that 
  \begin{enumerate}
   \item For all $c\in \op(k)$,  $d_s\in \op({j_s})$, and $x_t \in X$
    \[
     \theta(\gamma(c;d_1,\ldots,d_k);x_1,\ldots,x_j) = 
     \theta(c;y_1,\ldots,y_k)
    \]
    where $y_s = \theta(d_s;x_{j_1+\ldots+j_{s-1}+1},
    \ldots,x_{j_1+\ldots+j_s} )$. \label{item_alg_assoc}
   \item  For all $x \in X$
    \[
      \theta(1;x)=x\text{ and }\theta(\ast)=\ast.
    \]
   \item  For all $c\in \op(k)$, $x_s \in X$, $\sigma\in \Sigma_k$
    \[
      \theta (c\sigma;x_1,\ldots,x_k) = 
      \theta(c;x_{\sigma^{-1}1},\ldots ,x_{\sigma^{-1}k}).
    \]
  \end{enumerate}

\begin{prop}\label{w_action}
For any $L \op$-algebra $Y$, the action of any $w \in \cW$ on $Y$ is a homotopy equivalence.
\end{prop}

\begin{proof}
As $w$ has a homotopy inverse $w^{-1}$ in $L\op$, the associativity in condition \ref{item_alg_assoc} above implies that we have a homotopy equivalence
\[\theta (w^{-1}, -) \circ \theta (w, -) \simeq \theta (1, -) \colon Y \to Y.\qedhere\]
\end{proof}

\begin{prop}\label{localization_alg}
  There is a functor 
  \[ 
    L (= L^{TH}_\cW) \colon \op\text{-algebras} \to L\op\text{-algebras} 
  \]
  with the following three properties.
  \begin{enumerate}
  \item For each $\op$-algebra $X$, there is a zig-zag of natural $\op$-algebra maps  between  $X$ and 
  $L X$.  
  \item If $X$ is  an $L\op$-algebra, there is a map $X \to LX$ which has a left inverse up to homotopy.
   \item 
  For any $\op$-algebra $X$, $L\op$-algebra $Y$, and map of $\op$-algebras $X\to Y$, there is a canonical map 
  $LX\to Y$ so that 
  \[
   \xymatrix{X\ar[r]\ar[d]&Y\\LX\ar[ur]}
  \] 
  commutes up to homotopy. 
  \end{enumerate}
\end{prop}

This result follows from a very general construction for monads associated to operads.  We now recall the relevant definitions needed for the construction.

\begin{defn}
  If $\op$ is an operad and $(X, \ast)$ is a based simplicial set,  the 
  {\bf free $\op$-\oalg{} on $X$} is 
  \[
    \fa{\mop}{X} \coloneqq  \coprod_{n\geq 0} \left(\op(n) 
    \times_{\Sigma_n} X^n\right)\big/ \sim 
  \]
  where $\sim$ is a base point relation generated by
  \[
    (\sigma_ic;x_1,\ldots ,x_{n-1}) \sim  (c;s_i(x_1,\ldots ,x_{n-1}))
  \]
  for all $c \in \op(n)$, $x_i \in X$, and $0 \leq i < n$ where 
  $\sigma_ic = \gamma(c,e_i)$ with 
  \[
    e_i =(1^i,\ast,1^{n-i-1})\in \op(1)^i \times \op(0) \times \op(1)^{n-i-1},
  \] 
  and $s_i(x_1,\ldots ,x_{n-1})=(x_1,\ldots ,x_i,\ast,x_{i+1},\ldots ,x_{n-1})$.
\end{defn}

This construction defines a monad in $\sset_*$ and associates a natural transformation of monads $\mop \to \mopt$ to a  map of operads  $\op \to  \opt.$ 
Hence, it defines a functor from operads in $\sset$ to monads in $\sset_*$. 

For a monad $\mop$ in the category $\sset_*$, an {\bf $\mop$-\malg{}}  is a pair $(X, \xi)$ consisting of a simplicial set $X$ 
and a simplicial  map $\xi\colon \fa\mop X \to  X$  that is unital and associative. 
A {\bf map of $\mop$-\malgs{}} is a map of simplicial sets $f\colon X \to Y$ where the following diagram commutes.
\[\xymatrix{
 \fa\mop X \ar[r]^{\fa\mop{f} }\ar[d]_{\xi_X}&\fa\mop Y\ar[d]^{\xi_Y}
  \\
 X\ar[r]^{f} & Y
}\]
The functor from operads to monads described above defines an isomorphism between the category of $\op$-\oalgs{} 
and the category of $\mop$-\malgs{}.

A {\bf $\mop$-functor} \cite[2.2, 9.4]{May} in a category $\cD$ is a functor $F \colon  \sset_* \to  \cD$ and a unital and associative natural transformation 
$\lambda \colon  F \mop \to F$.  
The {\bf bar construction} \cite[9.6]{May} for a monad $\mop$, $\mop$-\malg{} $X$ and $\mop$-functor $F$, 
denoted $B_\bullet(F, \mop, X)$, is the simplicial object in $\cD$ where  
\[B_q(F, \mop, X)\coloneqq F\mop^q X.\]
If $\cD$ is the category of simplicial sets, $B_\bullet(F, \mop, X)$ is a bisimplicial set.
We define 
\[
  B(F, \mop, X)\coloneqq |B_\bullet(F, \mop, X)|
\]
where the geometric realization is the diagonal  simplicial set.

Then \autoref{localization_alg} is a special case of the following  familiar and  more general result.  
We include a proof for completeness. 

\begin{prop}
  A map of operads $\op \to  \opt$ defines a functor 
  \[ 
    \otop{\opt}{(-)} \colon \op\text{-algebras} \to \opt\text{-algebras} 
  \]
  satisfying the following properties.
  \begin{enumerate}
  \item For each $\op$-algebra $X$, there is a zig-zag of natural $\op$-algebra maps  between  $X$ and  
  $\otop{\opt}{X} $.  
  \item If $X$ is  an $\opt$-algebra there is a map $X \to \otop{\opt}{X} $ which has a left inverse up to homotopy.
 \item For any $\op$-algebra $X$, $\opt$-algebra $Y$ and map of $\op$-algebras $X\to Y$, there is a canonical map $\otop{\opt}{X} \to Y$ so that 
  \[
    \xymatrix{X\ar[r]\ar[d]&Y\\
    \otop{\opt}{X} \ar[ur]}
  \]
  commutes up to homotopy. 
 \end{enumerate}
\end{prop}

\begin{proof}
  The map of operads, $\op \to  \opt,$  the associated natural transformation of monads, 
  $\mop \to  \mopt,$ and  the composite
  \[
    \mopt \mop \to \mopt \mopt \to \mopt,
  \]
  give $\mopt$ the structure of an $\mop$ functor. Thus, for an $\op$-algebra $X$,  we may define 
  \[
    \otop{\opt}{X} \coloneqq B(\mopt, \mop, X)
  \] 
  and observe that $\otop{\opt}{X}$ inherits the structure of a $\opt$-algebra from the copy of $\mopt$ in the bar construction.  Then there is a zig-zag of $\op$-algebra  maps 
  \[
    X\xleftarrow{\sim} B(\mop, \mop, X)\to B(\mopt, \mop, X)= \otop{\opt}{X}
  \]
  where the left hand map is the natural homotopy equivalence between $X$ and its free resolution, and the right hand map is induced by the map $\op\to \opt$.  
  This construction is natural in $X$, and hence it defines  a functor 
  \[\otop{\opt}{-}: \op-algebras \to \opt-algebras.\]
  
  If $X$ is a $\opt$-algebra, consider the composite 
  \[
    \otop{\opt}{X} \coloneqq  B(\mopt  , \mop, X)
    \rightarrow B(\mopt, \mopt , X) \rightarrow X.
  \]  
  Then 
  \[ 
    B(\mop, \mop, X) \rightarrow  B(\mopt , \mop, X) \rightarrow B(\mopt, \mopt , X) \rightarrow X
  \]
  is the natural homotopy equivalence $B(\mop, \mop, X)\xrightarrow{\sim} X$. 

  For a map $X \to Y$ of $\op$-algebras, naturality gives a commutative diagram
  \[
    \xymatrix{X\ar[r]\ar[d]&Y\ar[d]\\
    \otop{\opt}{X}\ar[r]&\otop{\opt}{Y}}
  \]  
  As $Y$ is a $\opt$-algebra, the right vertical map has a left inverse.  
  Then the composite \[ \otop{\opt}{X}\to \otop{\opt}{Y}\to Y\] makes the following diagram commute up to homotopy
  \[\xymatrix{X\ar[r]\ar[d]&Y\\
\otop{\opt}{X}\ar[ur]}\]
\end{proof}

\bibliography{inverting}{}
\bibliographystyle{plain}
\end{document}